\documentclass[a4paper,11pt]{article}
\pdfoutput=1
\title{Rational points on non-linear horocycles and pigeonhole statistics for the fractional parts of $\sqrt{n}$} 
\author{Sam Pattison}

\usepackage{amsmath}
\usepackage{amssymb}
\usepackage{amsfonts}
\usepackage{enumerate}
\usepackage{amsthm}
\usepackage{amsfonts}
\usepackage{graphicx}
\usepackage{url}
\usepackage{dsfont}
\usepackage{bm}
\usepackage{mathrsfs}
\usepackage{comment}
\usepackage{graphicx}
\usepackage{geometry}

\graphicspath{ {./Figures/} }

\numberwithin{equation}{section}

\newtheorem{theorem}{Theorem}[section]
\newtheorem{lemma}{Lemma}[section]
\newtheorem{proposition}{Proposition}[section]
\newtheorem{corollary}{Corollary}[section]

\newtheorem{remark}{Remark}[section]

\usepackage[
backend=biber,
style=alphabetic,
sorting=ynt
]{biblatex}
\renewbibmacro{in:}{}

\begin{document}
\maketitle
\begin{abstract}
In this paper we investigate \textit{pigeonhole statistics} for the fractional parts of the sequence $\sqrt{n}$. 
Namely, we partition the unit circle $ \mathbb{T} = \mathbb{R}/\mathbb{Z}$ into $N$ intervals and show that the proportion of intervals containing exactly $j$ points of the sequence $(\sqrt{n} + \mathbb{Z})_{n=1}^N$ converges in the limit as $N \to \infty$.

More generally, we investigate how the limiting distribution of the first $sN$ points of the sequence varies with the parameter $s \geq 0$. A natural way to examine this is via point processes - random measures on $[0,\infty)$ which represent the arrival times of the points of our sequence to a random interval from our partition.
We show that the sequence of point processes we obtain converges in distribution and give an explicit description of the limiting process in terms of random affine unimodular lattices.

Our work uses ergodic theory in the space of affine unimodular lattices, building upon work of Elkies and McMullen.
We prove a generalisation of equidistribution of rational points on expanding horocycles in the modular surface,
working instead on non-linear horocycle sections.
\end{abstract}

\section{Introduction}\label{intro}
Let $\mathbb{T} := \mathbb{R}/\mathbb{Z}$ denote the circle, $\mathbb{N} := \{1,2,3, \dotsc \}$ be the set of natural numbers, $\mathbb{N}_0 := \mathbb{N} \cup \{0\}$ be the set of non-negative integers and $\mathbb{R}^+$ the set of non-negative real numbers. We investigate \textit{pigeonhole statistics} for the sequence $\sqrt{n}$ modulo 1. Specifically, we look at
the limiting distribution of the numbers $(\sqrt{n}  + \mathbb{Z})_{n=1}^{N}$ among partitions of $\mathbb{T}$ into intervals of length $\frac{1}{N}$ as $N \to \infty$.

For $s \geq 0$, $x_0 \in [0,1)$ and $N \in \mathbb{N}$ with $N \geq 1$ define
\begin{equation}\label{S_N}
    S_N(x_0,s):= |\{ 1 \leq n \leq sN: \sqrt{n} \in [x_0 - \tfrac{1}{2N} , x_0 + \tfrac{1}{2N}) + \mathbb{Z} \}|.
\end{equation}
When $x_0$ ranges over the set $ \Omega_N : =\{ \tfrac{k}{N}: 0 \leq k \leq N-1\} \subset [0,1)$ the $N$ intervals $[x_0 - \tfrac{1}{2N} , x_0 + \tfrac{1}{2N}) + \mathbb{Z}$ will partition $\mathbb{T}$ and so the average value of $S_N(x_0,s)$ as $x_0$ ranges over $\Omega_N$ will be $\frac{\lfloor sN \rfloor }{N} = s + O(\tfrac{1}{N})$. As a result it is natural to investigate the long term statistical properties of the sequences $\{S_N(x_0,s) : x_0 \in \Omega_N \}$ as $N \to \infty$ and, in particular, the proportion of terms equal to a given $j \in \mathbb{N}_0$ as $N \to \infty$. 
Indeed, for each $j \in \mathbb{N}_0$ we define 
\begin{equation}\label{E_N} 
E_{j,N}(s) := \frac{1}{N} |\{ 0 \leq k \leq N-1: S_N(\tfrac{k}{N},s) = j \}|.
\end{equation}
This is the proportion of the intervals $\{[x_0 - \tfrac{1}{2N} , x_0 + \tfrac{1}{2N}) + \mathbb{Z} :  x_0 \in \Omega_N \}$ containing exactly $j$ of the points $\{ \sqrt{n}  : 1 \leq n \leq sN \}$. Here we show:
\begin{theorem}\label{secondPigeon}
For all $j \in \mathbb{N}_0$ and $s \geq 0$, $E_j(s) := \lim_{N \to \infty} E_{j,N}(s)$ exists. Moreover, the limiting distribution function $E_j(s)$ is $C^2$ with respect to $s$. 
\end{theorem}

Our proof of Theorem \ref{secondPigeon} builds upon the work of Elkies and McMullen in \cite{elkies2004gaps}. Here ergodic theory and, specifically, Ratner's theorem are used to determine the gap distribution of the sequence $( \sqrt{n} + \mathbb{Z} )_{n=1}^{\infty}$ via relating these properties to the equidistribution of a family of closed orbits of a certain unipotent flow in the homogeneous space 
\begin{equation}\label{firstX}
    X = (\text{SL}(2,\mathbb{Z}) \ltimes \mathbb{Z}^2) \backslash (\text{\text{SL}}(2,\mathbb{R})\ltimes \mathbb{R}^2).
\end{equation}
We elaborate on this further in \textsection \ref{Ergodic Theory}.

\begin{remark}
\textup{
The limiting functions $E_j(s)$ are given more concretely by (\ref{LimitEquation}). They give the probabilitity the lattice corresponding a randomly chosen point $x \in X$ contains exactly $j$ points in a fixed triangle of area $s$ in the plane. The functions $E_j(s)$ agree with the limiting distribution for the probability of finding $j$ of the points of the sequence $\{\sqrt{n}+ \mathbb{Z} : 1 \leq n \leq sN \}$ in a randomly shifted interval of length $\frac{1}{N}$ in $\mathbb{T}$. (\cite{elkies2004gaps}).
They also agree with the limiting functions found by Marklof and  Str\"{o}mbergsson for the probability of finding exactly $j$ lattice points of a typical (2-dimensional) affine unimodular lattice in a ball of radius $N$ whose directions all lie in a random open disc of radius proportional to $\tfrac{s}{N^2}$ on the unit circle. (\cite[Theorem 2.1 and Remark 2.3]{marklof2010distribution}). As we will see in \textsection\ref{Pigeonhole Stats}, the work of Marklof and Str\"{o}mbergsson allows us to immediately infer the aforementioned differentiability of the limiting distribution functions. }
\end{remark}

\begin{remark}
\textup{
We do not give exact formulas for the functions $E_j(s)$ in terms of explicit analytic functions in this paper. The analogous functions for rectangles were considered by Str\"ombergsson and Venkatesh in \cite{strombergsson2005small}
who obtained explicit piecewise analytic formulas for small $j$. Based on their work, we would expect the functions $E_j(s)$ to be piecewise analytic with the functions becoming increasingly complex as $j$ increases.}
\end{remark}

\begin{remark}
\textup{
As is discussed in, for example, \cite{technau2020correlations}, the sequence of fractional parts of the sequence $\sqrt{n}$ is of interest from the point of view of \textit{fine scale statistics}. The gap distribution of this sequence in not Poissonian (see also Remark \ref{PoiRem}) which contrasts with the conjectured gap distribution of the fractional parts of $n^{\alpha}$ for any other $\alpha \in (0,1) \setminus \{  \tfrac{1}{2} \}$.
In our case, if we instead considered the fractional parts of $n^{\alpha}$ for $\alpha \in (0,1) \setminus \{  \tfrac{1}{2} \}$ we would expect Poissonian pigeonhole statistics in the sense that the corresponding limiting distribution functions $E_j(s)$ would equal $\tfrac{s^je^{-j}}{j!}$. This contrasts with the case $\alpha = \tfrac{1}{2}$ as shown in Figure \ref{Pigeonholegraphs}.
} 
\end{remark}

\begin{figure}[h]\label{Pigeonholegraphs}
\includegraphics[scale=0.53]{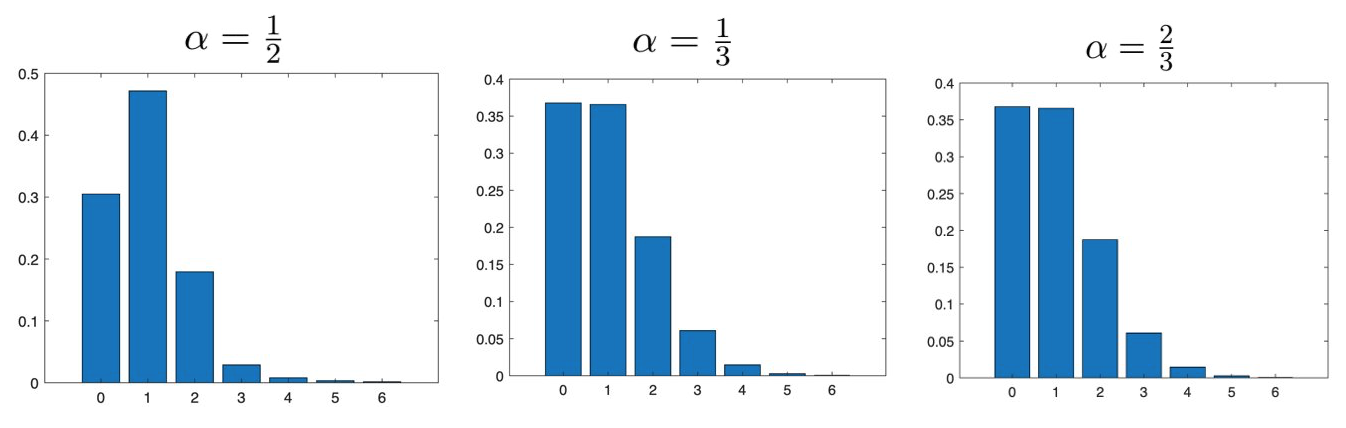}
Figure \ref{Pigeonholegraphs}: The proportion of $N = 10000000$ intervals in partition of $\mathbb{T}$ containing $0  \leq j \leq 6$ points of $n^{\alpha} + \mathbb{Z}$ for $n \leq N$ when $\alpha$ is equal to $\tfrac{1}{2}$, $\tfrac{1}{3}$ and $\tfrac{2}{3}$. For $\alpha =1/2$, these proportions approximate $E_j(s)$ for $s = 1$.
\end{figure}

We can also recast our problem in a probabilistic setting. Indeed, for $N \in \mathbb{N}$, let $W_N$ be a random variable which is distributed uniformly on the set $\Omega_N$. Then, we define a sequence of stochastic processes $Y^N_s$ for $N \in \mathbb{N}$ and $s \geq 0$ by setting 
\begin{equation}\label{Y_s^N}
    Y^N_s := S_N(W_N,s).
\end{equation}
With this notation Theorem \ref{secondPigeon} states that the sequence $\mathbb{P}(Y^N_s = j)$ converges as $N \to \infty$. 

For each $N$, we can also think about each point $x_0 \in \Omega_N$ as giving us a locally finite Borel measure on $\mathbb{R}^+$ of the form
$$ \eta_N(x_0) := \sum_{r=1}^{\infty} \delta_{s_r(x_0)} $$
where $s_1(x_0) < s_2(x_0) < s_3(x_0), \dotsc $ are the complete sequence of points $s \in \tfrac{1}{N}\mathbb{N}$ such $\sqrt{sN} \in [x_0- \tfrac{1}{2N}, x_0 + \tfrac{1}{2N} ) + \mathbb{Z} $. Namely, these are points of discontinuity of the map $s \longmapsto S_N(x_0,s)$. In this case we have the relation
\begin{equation}\label{processIntervalRelation}
\eta_N(x_0)([0,s]) = S_N(x_0,s).
\end{equation}
Again, recasting this in a probabilistic setting, we define the corresponding sequence of random measures/point processes $\xi_N$ by setting 
\begin{equation}\label{definition of xi}
    \xi_N := \eta_N(W_N)
\end{equation}
Equation (\ref{processIntervalRelation}) above tells us that, for an interval $(a,b] \subset \mathbb{R}^+$, we have that the point process and stochastic process are related via
\begin{equation}\label{processIntervalRelation2}
\xi_N((a,b]) = Y^N_b - Y^N_a.
\end{equation}
In this setting, we establish the following convergence result which helps us to understand how the limiting distribution of the points of our sequence varies with $s$.
\begin{theorem}\label{mainPigeon}
The point process $(\xi_N)_{N=1}^{\infty}$ converges in distribution to a point process $\xi$.
\end{theorem}
The process $\xi$ is defined similarly to the processes $\xi_N$ as the sum of Dirac delta measures associated to the jump points of a stochastic process $Y_s:X \to \mathbb{R}$.  Here the space $X$ can be the thought of as the homogeneous space of all two-dimensional affine unimodular lattices (which we show explicitly in \textsection\ref{X}) and $Y_s(x)$ gives the number of points of the lattice associated to $x\in X$ within a certain triangle of areas $s$ in the plane. More concretely, if the lattice associated to $x \in X$ is $L \subset \mathbb{R}^2$ and 
\begin{equation}
    \tau(\infty) := \{ (u,v) \in \mathbb{R}^2: u \geq 0, -u\leq v \leq u \}
\end{equation}
 then
\begin{equation}\label{detlaSum}
\xi(x) = \sum_{(u,v) \in L \cup \tau(\infty) } \delta_{\sqrt{u}} .
\end{equation}
As we illustrate in \textsection \ref{Proporties}, $\xi$ is a simple, intensity 1 process which does not have independent increments.

\begin{remark}\label{PoiRem}
\textup{
The pigeonhole statistics we consider were previously studied by Weiss and Peres for the fractional parts of the sequence $2^n \alpha$ (as well as higher dimensional generalisations). In this case the analogous processes converge to a Poisson point process \cite{Weiss}. A Poisson point process is also (almost surely) the limiting process we would obtain if, instead of generating our point processes via considering how the points of the sequence $\sqrt{n}$ distribute among shrinking partitions of $\mathbb{T}$, we instead consider the analogous processes defined for a sequence of points in $\mathbb{T}$ generated by a sequence of i.i.d random variables which are uniformly distributed on $\mathbb{T}$ (\cite[ §VI.6]{feller1957introduction}). }
\end{remark}

Similarly to what is observed in \cite{el2015two}, even though our limiting processes isn't Poissonian, its second moment is nearly Poissonian with an error resulting from 
the fact that, asymptotically, $\sqrt{N}$ of the points $\{ \sqrt{n} + \mathbb{Z} : 1 \leq n \leq N \}$ are $0$.

\begin{corollary}\label{SecondMoment}
$$\mathbb{E}[|Y^N_s|^2] \to \sum_{j=0}^{\infty} j^2E_j(s)^2 + s =\int_X |Y_s|^2 \; dm_X + s =  s^2 + 2s  $$ as $N \to \infty$. In particular 
$$
\textup{Var}[Y^N_s] \to 2s
$$
as $N \to \infty$.
\end{corollary}

\begin{remark}
\textup{If we desire the (more satisfactory) convergence of the variance of the random variables $Y^N_s$ to those of $Y_s$, one has to avoid the escape of mass resulting from the term $0$ appearing regularly in the sequence of fractional parts of $\sqrt{n}$. This can be done via removing the terms $\sqrt{n}$ when $n$ is a square and, in this case, we would have Var$[Y^N_s] \to s$ which is the variance we would obtain if the limiting point process were Poissonian. We will also use this approach in the proof of Corollary \ref{SecondMoment}.  }
\end{remark}

\subsection{Ergodic Theory}\label{Ergodic Theory}

Let $G = \text{ASL}(2,\mathbb{R}) = \text{SL}(2,\mathbb{R})\ltimes \mathbb{R}^2$ be the affine special linear group of $\mathbb{R}^2$ with multiplication law defined by
$$(M,x)(M', x') = (MM',xM' + x'),
$$
where elements of $\mathbb{R}^2$ are viewed as row vectors. Let $ \Gamma = \text{SL}(2,\mathbb{Z}) \ltimes \mathbb{Z}^2$ be the discrete subgroup of $G$ consisting of elements with integer entries. As is discussed in \textsection\ref{X}, $\Gamma$ is a lattice in $G$, meaning we have a fundamental domain $\widetilde{\mathcal{F}}$ with finite volume (and hence, up to normalization, volume 1) under the Haar measure $m_G$ on $G$. By restricting $m_G$ to $\widetilde{\mathcal{F}}$ and projecting to $X$, we have a right invariant probability measure $m_X$ on the space $X: = \Gamma \backslash G$ which we call the Haar measure on $X$ (\cite[Proposition 9.20]{ErgWithNt}). Let 
$$
\Phi(t) := 
\Bigg( 
\begin{pmatrix}
e^{-\frac{t}{2}} & 0 \\
0 & e^{\frac{t}{2}}
\end{pmatrix}
, (0,0) \Bigg) 
$$ and
$$a(N) := \Phi(\log(N)).$$

As in (\cite{elkies2004gaps}), we shall be concerned with the equidistribution of points on certain \textit{horocycle sections} in the space $X$.  Here, a horocycle section is a function $\sigma:\mathbb{R} \to G$ of the form $$
\sigma(t) := 
\bigg(
\begin{pmatrix}
1 & 2t \\
0 & 1
\end{pmatrix}
, (x(t),y(t)) \bigg),
$$
where $x(t)$ and $y(t)$ are smooth functions. We call $\sigma(t)$ a horocycle section of period $p \in \mathbb{N}$ if there exists some $\gamma_0 \in \Gamma$ such that $\gamma_0 \sigma(t + p) = \gamma_0 \sigma(t)$ for all $t \in \mathbb{R}$. Moreover, such a horocycle section is \textit{non-linear} if there exists some $\alpha, \beta \in \mathbb{Q}$ such that the set $\{ t \in [0,p] : y(t) = \alpha t + \beta \}$ has zero Lebesgue measure. For such horocycle sections, the following equidistribution result is known.

\begin{theorem}[\cite{elkies2004gaps} Theorem 2.2, \cite{marklof2007distribution} Theorem 4.2] \label{known}
Let $\sigma$ be a non-linear horocycle section with period $p$. Then, for any bounded continuous function $f:X \to \mathbb{R}$
$$
\frac{1}{p}\int_0^p f(\Gamma \sigma(x_0) \Phi(t)) \; dx_0 \to \int_X f \; dm_X
$$
as $t \to \infty$. 
\end{theorem}

Applying this to the non-linear period 1 horocycle section
$$
n(t) := 
\bigg(
\begin{pmatrix}
1 & 2t \\
0 & 1
\end{pmatrix}
, (t,t^2) \bigg),
$$
one can determine the distribution of $\{ \sqrt{n}\}_{n=1}^N$ among the intervals $[x_0 - 1/2N, x_0 + 1/2N) + \mathbb{Z}$ when $x_0$ is uniformly distributed on $[0,1)$. In our setting, we restrict $x_0$ to lying in the set $\Omega_N$ for each $N$ and the corresponding equidistribution we desire is that of rational points on such a horocycle section. We therefore prove the following result which, like Theorem \ref{known}, applies more generally to functions $f:X \to \mathbb{R}$ which are piecewise continuous: functions $f:X \to\mathbb{R}$ whose  points of discontinuity are contained in a set of measure zero with respect to $m_X$.

\begin{theorem}\label{main}
Let $\sigma$ be a non-linear horocycle section with period $p$. Then,
for any bounded piecewise continuous function $f:X \to \mathbb{R}$ and $C \geq 1$,
\begin{equation}\label{disceteIntf}
\frac{1}{pN} \sum_{k=0}^{pN-1} f(\Gamma \sigma( \tfrac{k}{N})a(M) ) \to \int_X f \; dm_X
\end{equation}
as $N \to \infty$ and $\tfrac{1}{C}N \leq M \leq CN$. 
\end{theorem}
As we show concretely in \textsection\ref{Pigeonhole Stats}, for an appropriate $f:X \to \mathbb{R}$ we can approximate $\mathbb{P}(\xi_N((a,b]) = 0)$ (or more generally $ \mathbb{P}(\xi_N(B) = 0)$ for $B$ as in Lemma \ref{processConvergence} (ii)) by a sum of the above form in (\ref{disceteIntf}) and, using the above equidistribution result, show Theorem \ref{mainPigeon}. The same principle applies in the case of Theorem \ref{secondPigeon}.

\begin{remark}
\textup{Although Theorems  \ref{secondPigeon}, \ref{mainPigeon} and \ref{main} are stated for the points/interval centres $\tfrac{k}{N}$ for $0 \leq k \leq N-1$, one can see the methods presented in this paper also give the analogous results when considering the points/interval centres $\tfrac{k+ \alpha}{N}$ for any $\alpha \in \mathbb{R} $. The choice $\alpha = \tfrac{1}{2}$ in particular results in considering the points of the sequence $\sqrt{n}$ in the intervals formed via partitioning by cutting $\mathbb{T}$ at the points $\tfrac{k}{N}$ for $0 \leq k \leq N-1$.}
\end{remark}

\begin{remark}
\textup{There are many known results related to Theorem \ref{main} when considering the equidistribution of discrete collections of points on expanding horocycle orbits. An effective equidistribution theorem for rational horocycle points $\{ k/N + iy \}_{k=0}^{N-1}$ in the modular surface is proved by Burrin, Shapira and Yu in \cite[Theorem 1.1]{burrin2020translates}. Using Spectral methods, \cite{burrin2020translates} shows such points equidistribute when the number of such rational points $N$ being considered at height $y$ satisfies $N \gg y^{-(\tfrac{39}{64} + \epsilon)}$ for some $\epsilon > 0$. This contrasts this with Theorem \ref{main} which corresponds to the case when $N \asymp y^{-1}$. Using Dynamical methods, Einsiedler, Luethi and Shah prove effective equistribution results for the  rational points
$$
\bigg\{
\bigg(\text{SL}(2,\mathbb{Z}) \begin{pmatrix}
1 & k/N \\
0 & 1
\end{pmatrix}
\begin{pmatrix}
N^{-1/2} & 0 \\
0 & N^{1/2}
\end{pmatrix}, \frac{k}{N} + \mathbb{Z}\bigg)
: 0 \leq K \leq N-1 \bigg\}
$$
in the more general space SL$(2,\mathbb{Z}) \backslash$SL$(2,\mathbb{R}) \times \mathbb{T}$ (\cite{einsiedler2019primitive}). The equidistribution of such points when projected SL$(2,\mathbb{Z}) \backslash$SL$(2,\mathbb{R})$ is implied by Theorem \ref{main}.
 Finally, in \cite{marklof2003equidistribution}, Marklof and Str\"{o}mbergsson prove for fixed $\delta > 0$, there is full measure set of $\alpha \in [0,1)$ such that the points $\{m\alpha + iy\}_{m=1}^N$ equidistribute in the modular surface as $y \to 0$ whenever $y \asymp N^{-\delta}$.
 }

\end{remark}

\subsection{Outline of Proof}\label{sketch}
Recall the following conditions which are sufficient to give the convergence in distribution of a sequence of point processes.
\cite[Theorem A2.2]{leadbetter1983extremes}

\begin{lemma}[{{\cite[Theorem A2.2]{leadbetter1983extremes}}}]\label{processConvergence}
Let $(\xi_n)_{n=1}^{\infty}$ and $\xi$ be point processes defined on $ \mathbb{R}^+$ with $\xi$ being simple. Suppose 
\begin{enumerate}[(i)]
\item $\mathbb{E}[ \xi_N((a,b]) ] \to \mathbb{E}[ \xi((a,b]) ] $ as $N  \to \infty$ for all $0 \leq a < b < \infty$.
\item $ \mathbb{P}[\xi_N(V) = 0 ] \to  \mathbb{P}[(\xi(V) = 0 ] $ for all $V$ of the form $ \cup_{j=1}^k (a_j,b_j]$ with $0 \leq a_1 < b_1 \leq a_2 < b_2 \leq \dotsc \leq a_k < b_k$. 
\end{enumerate}
Then $\xi_N \xrightarrow{d} \xi$, where $\xrightarrow{d}$ denotes convergence in distribution.
\end{lemma}
For our processes $\xi_N$ defined by (\ref{definition of xi}), we will see that condition (i) merely amounts to the fact the average number of point of an affine unimodular lattice points in a triangle of area $s$ is $s$. We prove this more generally in Lemma \ref{Integral0}.

Turning to (ii), we define the measures $(\nu_N)_{n=1}^{\infty}$ on $X$ by 
\begin{equation}
\nu_N(f) = \int_X f \; d\nu_N := \frac{1}{N} \sum_{k=0}^{N-1} f(\Gamma n( \tfrac{k}{N})a(N) ).
\end{equation}
In \textsection\ref{Pigeonhole Stats}, for a given set $V$ as in Lemma \ref{processConvergence} (ii), we show how to choose the function $f :X \to \mathbb{R}$ such that $\nu_N(f)$ approximates $\mathbb{P}(\xi_N(B) = 0)$. The same is true in proving Theorem \ref{secondPigeon}, where we choose a function $f:X \to \mathbb{R}$ such that $\nu_N(f)$ approximates $\mathbb{P}(Y^N_s = j)$. By taking $N \to \infty$ we can then show the required limiting values are attained using Theorem \ref{main}. For the remainder of this section, we thus focus on the proof of Theorem \ref{main}.

\begin{proof}[Proof outline of Theorem \ref{main} for $\sigma(t) = n(t)$]
By a standard approximation argument, it suffices to show we have  $\nu_N(f) \to \int_X f \; dm_X$ for all $f \in C_c(X)$. This allows us to reduce to understanding weak-star limit points of the sequence of measure $(\nu_N)$. In particular it suffices, by the Banach-Alaoglu Theorem, to show any accumulation point $\nu$ of the measures $(\nu_N)$ is $m_X$. 

As is shown in Proposition \ref{u(1)Inv}, moving from $n(\tfrac{k}{N})a(N)$ to $n(\tfrac{k+1}{N})a(N)$ corresponds, up to some negligible error, to right multiplication by the unipotent element $u(1)$, where
\begin{equation}\label{u(t)}
u(t) := 
\bigg(
\begin{pmatrix}
1 & 2t \\
0 & 1
\end{pmatrix}
, (0,0) \bigg)
\end{equation}
for $t \in \mathbb{R}$. It will follow that any such $\nu$ is invariant under the action of the subgroup $\{u(k)\}_{k \in \mathbb{Z}}$. 

The right action of this subgroup on $X$ is mixing, as is shown in Lemma \ref{Mixing}. A consequence of this is that the system $(X,U_t,m_X)$, where $U_t(x) = xu(t)$, is disjoint from the linear rotation flow on $[0,1)$ in the sense introduced by Furstenberg in \cite{furstenberg1967disjointness} (as is shown in Lemma \ref{Disjoint}). To be precise, the linear rotation flow $R_t:[0,1) \to [0,1)$ is given by $R_t(s) = \{s+t\}$, where $\{ \cdot \}$ gives the fractional part of a real number. This is used to extend to the flow $U_t$ to $\widetilde{U}_t: X \times [0,1)\to X \times [0,1)$ given by $\widetilde{U}_t(\Gamma g,s) = (\Gamma g u(t),\{s+t\})$. Disjointness then tells us that
the only $\widetilde{U}_t$-invariant measure on $ X \times [0,1) $ whole marginals (projections to $X$ and $[0,1)$) are $m_X$ and the Lebesgue measure $ds$ on $[0,1)$ is the product measure $m_X \times ds$.

To utilise this fact, we consider the corresponding \textit{special flow under the ceiling function} $1$: namely the flow $T_t:X \times [0,1) \to X \times [0,1)$ given by 
$$T_t(\Gamma g,s) = (\Gamma g u(\lfloor s + t \rfloor ), \lbrace s+t \rbrace ).$$
This flow has $\nu \times ds$ as an invariant measure and is also conjugate to the flow $\widetilde{U}_t$. Keeping track of the measure $\nu \times ds$ under this conjugation map (described explicitly in the proof of Proposition \ref{special flow}) and using Theorem \ref{known}, we see the resulting measure on $X \times [0,1)$ indeed has marginals $m_X$ and $ds$ and so is the product measure $m_X \times ds$. This in turn gives us that $\nu \times ds = m_X \times ds$ by applying the inverse of the conjugation map and so $\nu = m_X$ as required.
\end{proof}

\section{The Space $X$}\label{X}
Here we overview, for completeness, some of the basic properties of the space $X = \Gamma \backslash G$ which we will be using. More details can be found in \cite[\S 3.1]{marklof2007distribution} and \cite[\S 1]{strombergsson2015effective}.

\begin{itemize}
    \item $X$ is a $\mathbb{T}^2 = \mathbb{R}^2 / \mathbb{Z}^2 $ bundle over the base space $B :=\text{SL}(2,\mathbb{Z}) \backslash \text{SL}(2,\mathbb{R})$. If $\mathcal{F}$ is a fundamental domain for the left-action of SL$(2,\mathbb{Z})$ on SL$(2,\mathbb{R})$, then a fundamental domain for the left-action of $\Gamma$ on $G$ is 
    \begin{align*}
        \widetilde{\mathcal{F}} =&  \{ (I_2,x)(M,0) : x \in [0,1), M \in \mathcal{F} \} \\ =& \{ (M,x) \in G : M \in F \text{ and } x \in [0,1)^2M \}.
    \end{align*}
    We fix such $\mathcal{F}$ and $\widetilde{\mathcal{F}}$ for the remainder of paper.
    \item Let $m_{\text{SL}(2,\mathbb{R})}$ be the Haar measure on the unimodular group $\text{SL}(2,\mathbb{R})$, normalised so that $m_{\text{SL}(2,\mathbb{R})}(\mathcal{F}) = 1$. Using Fubini's theorem and the translation invariance, it is easy to see $m_G = m_{\text{SL}(2,\mathbb{R})} \times dx$ is a (left) Haar measure on X, where $dx$ represents the Lebesgue measure on $\mathbb{R}^2$. The right invariant measure $m_X$ on $X$ is obtained by then restricting this measure $m_G$ to $\widetilde{\mathcal{\mathcal{F}}}$.
    \item There exists a left-invariant Reimannian metric $d_G$ on $G$ inducing the same topology on $G$ as the product topology on the space SL$(2,\mathbb{R}) \times \mathbb{R}^2$. Fixing one such metric $d_G$, we construct a metric $d$ on $X$ via defining 
    $$
    d(\Gamma g_1, \Gamma g_2) := \inf_{\gamma \in \Gamma} d_G(\gamma g_1,g_2).
    $$
    For a more explicit details on these constructions, see \cite[\S 9.3]{ErgWithNt}. Throughout the remaining sections, continuity of functions $f:X \to \mathbb{R}$ will mean continuity with respect to this metric. 
    \item Any element $(M, x) \in G$ gives us an affine unimodular lattice in $\mathbb{R}^2$ - namely the lattice
    $\mathbb{Z}^2M + x$. Moreover, for any other $(M', x') \in G$, the lattice associated to $(M',x')(M, x)$ is given by $\mathbb{Z}^2M'M + x'M + x$. These two lattices are identical if and only if $(M' ,x') \in \Gamma$. Thus we have a natural identification between elements of $X = \Gamma \backslash G$ and such lattices. We will use this identification in \textsection\ref{Pigeonhole Stats} to construct the functions $f:X \to \mathbb{R}$ to which we will apply Theorem \ref{main}.
\end{itemize}

\section{The Special flow under 1}\label{The Special Flow under 1}

Throughout this section, whenever $(X_1, \mu_1)$ is a measure space, $X_2$ is a measurable space and $\mathcal{T}:X_1 \to X_2$ is a measurable map, we will define the measure $\mathcal{T}_{*}\mu_1$ on $X_2$ by 
$$
\mathcal{T}_{*}\mu_1(A) = \mu_1(\mathcal{T}^{-1}(A))
$$
for any measurable $A \subset X_2$.

\begin{proposition}\label{u(1)Inv} 
Let $\sigma$ be a non-linear horocycle section of period $p$. Define the measures
$(\nu_N)_{n=1}^{\infty}$ on $X$ by setting 
\begin{equation}\label{nu}
\nu_N(f) = \int_X f \; d\nu_N := \frac{1}{pN} \sum_{k=0}^{pN-1} f(\Gamma \sigma( \tfrac{k}{N})a(N) )
\end{equation}
for any continuous bounded $f:X \to \mathbb{C}$. 
Then, any weak-star limit point of the measures defined in (\ref{nu}) is invariant under the map $T:X \to X$ given by 
\begin{equation}\label{T}
T(\Gamma g) = \Gamma g u(1)
\end{equation}
where $u(1)$ is defined by (\ref{u(t)}).

\end{proposition}

To see this, we will need the following Lemma.

\begin{lemma}\label{uniformcontinuity}
Any $f \in C_c(X)$ is uniformly continuous in the $\mathbb{T}^2$ direction. More precisely, for any $\epsilon > 0$ we can find $\delta > 0$ such that for all $M \in $SL$(2,\mathbb{R})$ and $u,v \in \mathbb{T}^2$ with $d_{\mathbb{T}^2}(u,v) \leq \delta$, we have 
$$
| f((I_2,u)(M,0)) - f((I_2,v)(M,0)) | < \epsilon.
$$
\begin{proof}
Take $f \in C_c(X)$ and let $K$ be the projection of the support of $f$ to the base space $B $. $K$ is a compact set and so the map $K \times \mathbb{T}^2 \ni (M,x) \longmapsto f((I_2,x)(M,0)) \in \mathbb{R}$ is uniformly continuous. This means $f$ is uniformly continuous in the fibre direction over $K$ in the sense that for any $\epsilon > 0$ we can find $\delta > 0$ such that for any $M \in K$ and $u,v \in \mathbb{T}^2$ with $d_{\mathbb{T}^2}(u,v) < \delta$ we have 
\begin{equation*}
    | f((I_2,u)(M,0)) - f((I_2,v)(M,0)) | < \epsilon.
\end{equation*}
Hence, since $f$ is identically zero on the fibre above all base points outside of $K$, $f$ is in fact uniformly continuous in the fibre direction over all of $B $.
\end{proof}

\end{lemma}

\begin{proof}[Proof of Proposition \ref{u(1)Inv}]
Suppose $\nu$ is a weak-star limit of the sequence of measures $(\nu_{N_j})$ where $N_j \nearrow \infty$. 

Now, for any $N \in \mathbb{N}$ and $0 \leq k \leq pN-1$, we will see via (\ref{matrixeq1}) and (\ref{matrixeq2}) that the two points $ \sigma(\tfrac{k+1}{N})a(N)$ and $ \sigma(\tfrac{k}{N})a(N)u(1) $ are identical in their SL$(2,\mathbb{R})$ components and, as the functions $x$ and $y$ are smooth and so bounded and Lipschitz on $[0,p]$, differ by a distance $O(\tfrac{1}{N} )$ in the $ \mathbb{T}^2$ direction. Using this, we will see the measures $\{T_{*}\nu_{N_j}\}_{j \in \mathbb{N}}$ given by
$$
T_{*}\nu_{N_j}(f) =  \frac{1}{pN_j} \sum_{k=0}^{pN_j-1} f(\Gamma \sigma(\tfrac{k}{N_j})a(N_j)u(1)), \hspace{0.5cm} f \in C_c(X)
$$
will also converge to $\nu$ as $j \to \infty$ in the weak-star topology, since, for compactly $f \in C_c(X)$, $f(\sigma(\tfrac{k}{N_j})a(N_j)u(1))$ and $f(\sigma(\tfrac{k+1}{N})a(N))$ will be uniformly close across all $0 \leq k \leq pN_j-1$ (\ref{closeacrossk}). This will follow from the Lemma \ref{uniformcontinuity}.

Indeed, we have
\begin{align}
\sigma(\tfrac{k+1}{N})a(N) = &
\bigg(
\begin{pmatrix}
N^{-1/2}  & 2(k+1)N^{-1/2}  \\
0 & N^{1/2}
\end{pmatrix}
, \bigg(  x\big(\tfrac{k+1}{N})N^{-1/2} , y\big(\tfrac{k+1}{N}\big) N^{1/2} \bigg) \bigg) \nonumber \\
= & \big( I_2 , \big(x(\tfrac{k+1}{N}) , y(\tfrac{k+1}{N}) - \tfrac{2(k+1)}{N}x(\tfrac{k+1}{N}) \big) \big) \label{matrixeq1} \\
\times & 
\bigg(
\begin{pmatrix}
N^{-1/2}  & 2(k+1)N^{-1/2} \\
0 & N^{1/2}
\end{pmatrix}
, \big(0,0 \big) \bigg) \nonumber \\ \nonumber
\end{align}
and
\begin{align}
\sigma(\tfrac{k}{N})a(N)u(1) = &
\big( I_2 , \big(x(\tfrac{k}{N}) , y(\tfrac{k}{N}) - \tfrac{2k}{N}x(\tfrac{k}{N}) \big) \big) \label{matrixeq2} \\
\times & \bigg(
\begin{pmatrix}
N^{-1/2}  & 2(k+1)N^{-1/2}  \\
0 & N^{1/2}
\end{pmatrix}
, \big(0,0 \big) \bigg). \nonumber \\ \nonumber
\end{align}
 
Take $f \in C_c(X)$ and let $\epsilon > 0$. Choose $\delta > 0$ as given by Lemma \ref{uniformcontinuity} for such $\epsilon$. By the fact  $x,y$ are bounded and Lipschitz on $[0,p]$, for any sufficiently large $j$ sufficiently large we have that
\begin{equation}\label{closeacrossk}
    d_{\mathbb{T}^2}\big( 
 \big(x(\tfrac{k+1}{N_j}) , y(\tfrac{k+1}{N_j}) - \tfrac{k+1}{N_j}x(\tfrac{k+1}{N_j})  \big) ,
 \big(x(\tfrac{k}{N_j}) , y(\tfrac{k}{N_j}) - \tfrac{k}{N_j}x(\tfrac{k}{N_j}) \big)  \big) < \delta
\end{equation}
 for all $0\leq k \leq N_j-1$. Hence, for such $j$, 
\begin{align*}
& |T_{\ast}(\nu_{N_j})(f) - \nu_{N_j}(f) | \\
\leq & \frac{1}{pN_j}\bigg\vert \sum_{k=0}^{pN_j-1} f(\Gamma \sigma(\tfrac{k}{N_j})a(N_j)u(1)) - f(\Gamma \sigma(\tfrac{k}{N_j})a(N_j))  \bigg\vert \\
\leq & \frac{1}{pN_j}\bigg\vert \sum_{k=0}^{pN_j-1} f(\Gamma \sigma(\tfrac{k}{N_j})a(N_j)u(1)) - f(\Gamma \sigma(\tfrac{k+1}{N_j})a(N_j))  \bigg\vert + O\big(\tfrac{\parallel f \parallel_{\infty}}{N_j}\big) \\
 \leq & \epsilon + O\big(\tfrac{\parallel f \parallel_{\infty}}{N_j}\big)
\end{align*}
So $ \limsup_{j \to \infty} |T_{\ast}(\nu_{N_j})(f) - \nu_{N_j}(f) | \leq \epsilon$ for any $\epsilon > 0$ and so 
$ T_{\ast}(\nu) = \lim_{j \to \infty} \; T_{\ast}(\nu_{N_j})(f) = \lim_{j \to \infty} \; \nu_{N_j}(f) = \nu(f) $ as required. 
\end{proof}

Next, as mentioned in \textsection\ref{sketch}, we will use the special flow under the ceiling function $1$ to show that any weak-star limit point $\nu$ of the measures (\ref{nu}) is the Lebesgue measure. Specifically, the special flow will give us a system with invariant measure $\nu \times ds$ conjugate to a joining of the systems $(X,U_t, m_X)$ and $([0,1), R_t, ds)$, where $U_t(\Gamma g) = \Gamma g u(t) $ and $R_t(s) = \{ s + t \}$. This will imply $ \nu = m_X$ due the following.

\begin{lemma}\label{Disjoint}
The flows $(X,U_t, m_X)$ and $([0,1), R_t, ds)$ are disjoint. 
\end{lemma}

To see this, we will use the following Lemmas.

\begin{lemma}\label{Mixing}
The system $(X,U_t,m_X)$ is mixing.

\begin{proof}
This follows from applying the proposition from \cite[\textsection 2.2]{kleinbock1999badly} to the system $(X,U_t, m_X)$ (instead of a diagonal flow) and using that the horocycle flow on $B$ is ergodic.
\end{proof}

\end{lemma}

\begin{lemma}[{{\cite[Proposition 2.2]{Disjointness}}}]\label{ergodicdis}
Let $T:X \to X$ be an ergodic measure preserving transformation with respect to the measure $m_X$. Then, $(X,T,m_X)$ is disjoint from any measure preserving system given by the identity map $I:Y \to Y$ on a probability space $(Y,\mu)$.
\end{lemma}

\begin{proof}[Proof of Lemma \ref{Disjoint}]
Let $\mu$ be a joining of $(X,U_t, m_X)$ and $([0,1), R_t, ds)$: an invariant measure on $X \times [0,1)$ for the map $\widetilde{U}_t(\Gamma g,s) = (\Gamma g u(t),\{s+t\})$ whose marginals are $m_X$ and $ds$.
$\mu$ will be invariant under the map $\widetilde{U}_1$, and so is a joining of the systems $(X,U_1,m_X)$ and $([0,1),R_1,ds).$ $R_1 $ is the identity and, by Lemma \ref{Mixing}, $U_1$ is mixing and hence ergodic. Thus it follows from Lemma \ref{ergodicdis} that $\mu = m_X \times ds$ as required.
\end{proof}

We are now in a position to prove the following.

\begin{proposition}\label{special flow}
Any weak-star limit point $\nu$ of the measures (\ref{nu}) is the Haar measure $m_X$ on $X$. 
\end{proposition}

As mentioned, the main construction we will use in this proof is the special flow under the ceiling function 1.

\begin{lemma}[{{\cite[Lemma 9.23]{ErgWithNt}}}]\label{flowintro}
Let $\nu$ be a finite measure on $X$ which is invariant under $u(1)$. Then $\nu \times ds $ is an invariant measure for the map $
T_t:X \times [0,1) \to X \times [0,1)$ given by $$ T_t(\Gamma g,s) = (\Gamma g u(\lfloor s + t \rfloor ), \lbrace s+t \rbrace ).
$$
\begin{proof}
If $\nu(X) > 0$, the result is given by \cite[Lemma 9.23]{ErgWithNt} (which applies to probability measures and hence any non-zero finite measure via normalizing). Otherwise, the result is trivial as $\nu \times ds$ is the zero measure.
\end{proof}

\end{lemma}

\begin{proof}[Proof of Proposition \ref{special flow}]
Note that, for a weak-star limit point $\nu$ of the probability measures in (\ref{nu}), we have $\nu(X) \in [0,1]$. Thus, Lemma \ref{flowintro} implies $\nu \times ds$ is $T_t$ invariant, where $T_t$ is as in the statement of Lemma \ref{flowintro}. 

Now let $\psi: X \times [0,1) \to X \times [0,1) $ be given by 
$\psi(\Gamma g, s) = (\Gamma g u(s) , s)$ and recall the extension of the flow $U_t$ to $X \times [0,1)$ is given by $\widetilde{U}_t(x,s) = (xu(t),\{s+t\})$. Using that $s + t = \lfloor s + t \rfloor + \lbrace s + t \rbrace$ we see that $ \psi \circ T_t = \widetilde{U}_t \circ \psi$, meaning $T_t$ and $\widetilde{U}_t$ are conjugate via $\psi$ and $\mu := \psi_{\ast}(\nu \times ds) $ is an invariant measure for the flow $\widetilde{U}_t$. Denote the projection maps from $X \times [0,1)$ to $X$ and $[0,1)$ by $P_X$ and $P_{[0,1)}$ respectively. ${P_{[0,1)}}_{\ast}(\mu)$ is invariant under all $R_t: [0,1) \to [0,1) $ with $t \in \mathbb{R}$ and so, if it is a probability measure, it is the Lebesgue measure $ds$ on $[0,1)$.

We now show $({P_{X}})_{\ast}\mu $ is the Haar measure $m_X$ on $X$, which in turn shows $\mu$ is a probability measure. To do this, take $f \in C_c(X)$ and let $N_j \nearrow \infty$ be a sequence of natural numbers such that $\nu_{N_j}$ converges weak-star to $\nu$ as $j \to \infty$. Then
\begin{align}
\nonumber \int_X f \; d ({P_X})_{\ast}(\mu) & = \int f \circ P_X \circ \psi \; d (\nu \times ds)  \\
\nonumber & = \int_0^1 \int_X f(xu(s)) \; d \nu(x) ds   \\  
\nonumber & = \int_0^1 \lim_{j \to \infty} \frac{1}{pN_j} \sum_{k=0}^{pN_j -1} f(\Gamma \sigma(\tfrac{k}{N_j})a(N_j)u(s)) \; ds \\
& =  \lim_{j \to \infty} \int_0^1  \frac{1}{pN_j} \sum_{k=0}^{pN_j -1} f(\Gamma \sigma(\tfrac{k}{N_j})a(N_j)u(s)) \; ds, \label{bob} 
\end{align}
where the last equality follows from the dominated convergence theorem (as $f$ is bounded).

Similarly to as in the proof of Proposition \ref{u(1)Inv},
$ n(\tfrac{k+s}{N})a(N)
$
and
$
n(\tfrac{k}{N})a(N)u(s) 
$
have the same base point and are a distance at most $O(\tfrac{1}{N})$ apart in the $\mathbb{T}^2$ fibre direction whenever $s \in [0,1]$. Hence, by Lemma \ref{uniformcontinuity}, given any $\epsilon > 0$ we can ensure
\begin{equation}\label{steve}
    \int_0^1 \frac{1}{pN_j}\sum_{k=0}^{pN_j -1} f(\Gamma \sigma(\tfrac{k+s}{N_j})a(N_j) ) \; ds
\end{equation}
 and the integral in (\ref{bob}) differ by at most $\epsilon$ provided $j$ is sufficiently large. But, by making the substitution $t = (s+k)/N_j$ we see that
 \begin{align}
     \int_0^1 \frac{1}{pN_j}\sum_{k=0}^{pN_j -1} f(\Gamma \sigma(\tfrac{k+s}{N_j})a(N_j) ) \; ds & = \frac{1}{pN_j}\sum_{k=0}^{pN_j -1} \int_0^1 f(\Gamma \sigma(\tfrac{k+s}{N_j})a(N_j) ) \;  ds \nonumber \\
     & = \frac{1}{p} \sum_{k=0}^{pN_j -1} \int_{\tfrac{k}{N_j}}^{\tfrac{k+1}{N_j}} f(\Gamma \sigma(t) a(N_j)) \; dt \nonumber \\
     & =\frac{1}{p}\int_0^p f(\Gamma \sigma(t) a(N_j)) \; dt \label{charles}
 \end{align}
and, by Theorem \ref{known}, (\ref{charles}) converges to $\int f \; d m_X$ as $j \to \infty$. Hence, we have shown that for any $\epsilon > 0$
$$
 \bigg | \int_X f \; d ({P_X})_{\ast}(\mu) - \int_X f \; d m_X \bigg| \leq \epsilon.
$$
Therefore $({P_{X}})_{\ast}(\mu) = m_X$ and so $\mu$ is a joining of $(X,U_t, m_X)$ and $(\mathbb{T}, R_t, ds)$. Since Lemma \ref{Disjoint} shows these two systems are disjoint, we conclude $\mu = m_X \times ds$. To see finally that this implies $\nu = m_X$ note, since $m_X$ is invariant under the right-action of $G$, we have
\begin{align*}
\int g \; d(\nu \times ds) = \int g \; d \psi^{-1}_{\ast}(\mu) = \int _{0}^1 \int_X g(xu(-s),s) \; dm_X(x) ds \\
=  \int _{0}^1 \int_X g(x,s) \; dm_X(x) ds = \int g \; d(m_X \times ds)
\end{align*}
for any $g \in C_c(X \times \mathbb{T})$. Thus $\nu \times ds = m_X \times ds$ and so $\nu = m_X$.

\end{proof}

\section{Completing the Proof of Theorem \ref{main}}\label{Mixing of the flow u(t)}

\begin{proof}[Proof of Theorem \ref{main}]
By the by the Banach-Alaoglu Theorem, any subsequence of the measures $(\nu_N)$ defined in (\ref{nu}) have a further subsequence which converges weak-star to some limiting measure $\nu$. By Proposition \ref{special flow}, $\nu = m_X$. This shows the sequence of measures $(\nu_N)$ indeed converges weak-star to $m_X$.

Notice that for any constant function $f:X \to \mathbb{R}$, it is immediate that $\int_X f \; d\nu_N \to \int f \; dm_X$ as $N \to \infty$. This convergence therefore also holds for continuous functions which are constant outside of a compact set, being the sum of a constant function and a function in $C_c(X)$. Now, let $f:X \to \mathbb{R}$ be a bounded continuous function and let $\epsilon > 0$. Then we can find continuous functions $f_-, f_+ : X \to \mathbb{R}$, which are constant outside some compact set, with $f_- \leq f \leq f_+$ and for which 
$$
\int_X f_+ - f_- \; dm_X < \epsilon .
$$
Then we have
\begin{align*}
\int f \; dm_X - \epsilon \leq \int f_- \; dm_X = \liminf_{n \to \infty} \nu_N(f_-) \leq \liminf_{n \to \infty} \nu_N(f)  \\
\leq \limsup_{n \to \infty} \nu_N(f) \leq \limsup_{n \to \infty} \nu_N(f_+) = \int f_+ \; dm_X \leq \int f \; dm_X + \epsilon
\end{align*}
meaning
$$ 
 \liminf_{n \to \infty} \nu_N(f) =  \limsup_{n \to \infty} \nu_N(f) = \int f dm_X
$$ 
as our choice of $\epsilon > 0$ was general. Thus, $(\nu_N)$ converges weakly to $m_X$ and so $\nu_N(f) \to \int f \; dm_X$ as $N \to \infty$ for all piecewise continuous $f:X \to \mathbb{R}$ by the continuous mapping theorem.

Finally, let $C \geq 1$ and $(M_N)_{N=1}^{\infty}$ be a sequence satisfying $\tfrac{1}{C}N \leq M_N \leq CN$ for all $N \in \mathbb{N}$. Take an arbitrary subsequence $(M_{N_j})_{j=1}^{\infty}$ of the sequence $(M_N)$. By compactness of the interval $[\tfrac{1}{C},C]$ we can find a further subsequence of the $(M_{N_j})$, which we will still index by $N_j$, such that $\frac{M_{N_j}}{N_j} \to c \in [\tfrac{1}{C},C]$ as $j \to \infty$. Let $f \in C_c(X)$ and define $h \in C_c(X)$ by setting
$$
h(\Gamma g) := f(\Gamma g a(c)).
$$
Note that 
$$
\nu_{N_j}(h) = \frac{1}{pN_j}\sum_{k=0}^{pN_j-1} f(\Gamma \sigma(\tfrac{k}{N_j})a(cN_j)) .
$$
Using the metric $d$ defined in \S \ref{X}, we see 
$$
d( \Gamma \sigma(\tfrac{k}{N_j})a(cN_j), \Gamma \sigma(\tfrac{k}{N_j})a(M_{N_j})) \leq d_G\Big(e, a\Big(\tfrac{M_{N_j}}{cN_j}\Big)\Big) \to 0
$$
uniformly in $k$ as $j \to \infty$. Using this and the fact that, as $f$ is continuous and compactly supported, $f$ is uniformly continuous, we have
$$
\bigg| \frac{1}{pN_j}\sum_{k=0}^{pN_j-1} f(\Gamma \sigma(\tfrac{k}{N_j})a(cN_j)) - \frac{1}{pN_j}\sum_{k=0}^{pN_j-1} f(\Gamma \sigma(\tfrac{k}{N_j})a(M_{N_j})) \bigg |\to 0
$$
as $j \to \infty$. Thus, given $\nu_{N_j}(h) \to \int h \; dm_X$ and $\int_X h \; dm_X = \int_X f \; dm_X$ by the right invariance of $m_X$, we have
\begin{equation}\label{N_j}
\frac{1}{pN_j}\sum_{k=0}^{pN_j-1} f(\Gamma \sigma(\tfrac{k}{N_j})a(M_{N_j})) \to \int f \; dm_X
\end{equation}
as $j \to \infty$. Since our original subsequence was arbitrary, (\ref{N_j}) holds in the case where $N_j = j$ as required. This can  also be extended to any piecewise continuous function $f:X \to \mathbb{R}$ by the standard approximation argument above. 
\end{proof}

\section{Pigeonhole Statistics}\label{Pigeonhole Stats}
As mentioned in \textsection\ref{Ergodic Theory}, to prove Theorem \ref{processConvergence}, we are going to apply Theorem \ref{main} to a family of functions $f:X \to \mathbb{R}$ such that $\nu_N(f)$ gives us, up to some error of $o(1)$ in $N$, $\mathbb{P}(\xi_N((a,b]) = 0)$.

For a non-negative measurable function $f:\mathbb{R}^2 \to \mathbb{R}$, we define $\widehat{f}:X \to \mathbb{R}$ by setting $\widehat{f}(x)$ to be the sum of all the function values at the lattice points corresponding to $x \in X$. Explicitly
$$
\widehat{f}(\Gamma(M, x)) = \sum_{m \in \mathbb{Z}^2} f(mM + x).
$$
For such functions, the following simple version of Siegel's formula holds (\cite{siegel1945mean}). 

\begin{lemma}\label{Integral0}
Let $f:\mathbb{R}^2 \to \mathbb{R}$ be a non-negative measurable function. Then
$$
\int_X \widehat{f} \; dm_X = \int_{\mathbb{R}^2} f \; dx.
$$

\begin{proof}
Using the non-negativity of $f$, the fact SL$(2,\mathbb{R})$ consists of matrices on determinant 1 and the form of the Haar measure $m_X$ described in \S \ref{X}, we see 
\begin{align*}
    \int_X \widehat{f} \; dm_X =&  \int_{\mathcal{F}} \int_{[0,1)^2M} \widehat{f} \; dx \; dm_{\text{SL}(2,\mathbb{R})} \\
    = & \int_{\mathcal{F}} \int_{[0,1)^2} \sum_{m \in \mathbb{Z}^2} f((m + x)M) \; dx \; dm_{\text{SL}(2,\mathbb{R})}.
\end{align*}
The result then follows from the fact that
$$
\int_{[0,1)^2} \sum_{m \in \mathbb{Z}^2} f((m + x)M) \; dx = \int_{\mathbb{R}^2} f(x) \; dx.
$$
\end{proof}

\end{lemma}

For a set $A \subset \mathbb{R}^2$ we denote by $f_A:X \to \mathbb{R}$ the function $\widehat{\chi_A}$, where $\chi_A$ denotes the indicator function of the set $A$. 
Following \cite[\textsection4]{marklof2007distribution}, we see how such functions can be used to approximate the values of the functions $S_N$. This will allow us to show the random variables $Y_s^N$ defined by (\ref{Y_s^N}) converge to the same limit of a sequence of random variables $\widetilde{Y}_s^N$ which will be defined by evaluating such a function $f_A$ at the points $n(k/N)a(N)$ uniformly at random.

Indeed, fixing some $s > 0$ and setting $N' = \lfloor sN \rfloor$, the counting function $S_N(x_0,s)$ defined in (\ref{S_N}) is given by
\begin{equation}\label{couting}
S_N(x_0,s) = \sum_{n=1}^{N'} \sum_{m \in \mathbb{Z}} \chi_{[-\tfrac{1}{2} ,\tfrac{1}{2} )} ( N(\sqrt{n} - x_0 + m) ).
\end{equation}
It turns out $S_N(x_0,s) $ can be well approximated by $f_{\tau}(n(x_0)a(N))$, where 
\begin{equation}
    \tau = \tau(s): =  \{(x,y) \in \mathbb{R}^2 : x \in [0,\sqrt{s}], y \in [-x,x] \}
\end{equation}
is a triangle of area $s$ in the plane (see Figure \ref{Shot1}).

\begin{figure}\label{Shot1}
\includegraphics[scale=0.35]{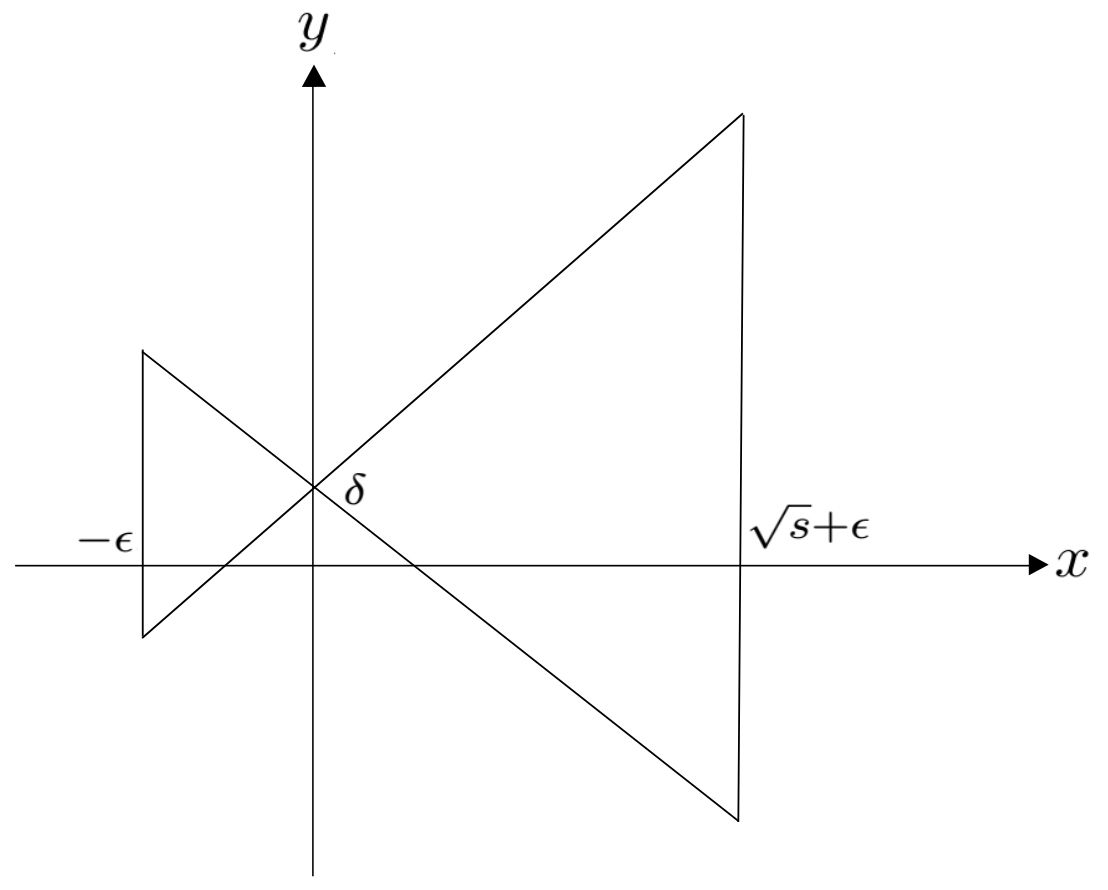} \hspace{1cm} \includegraphics[scale=0.35]{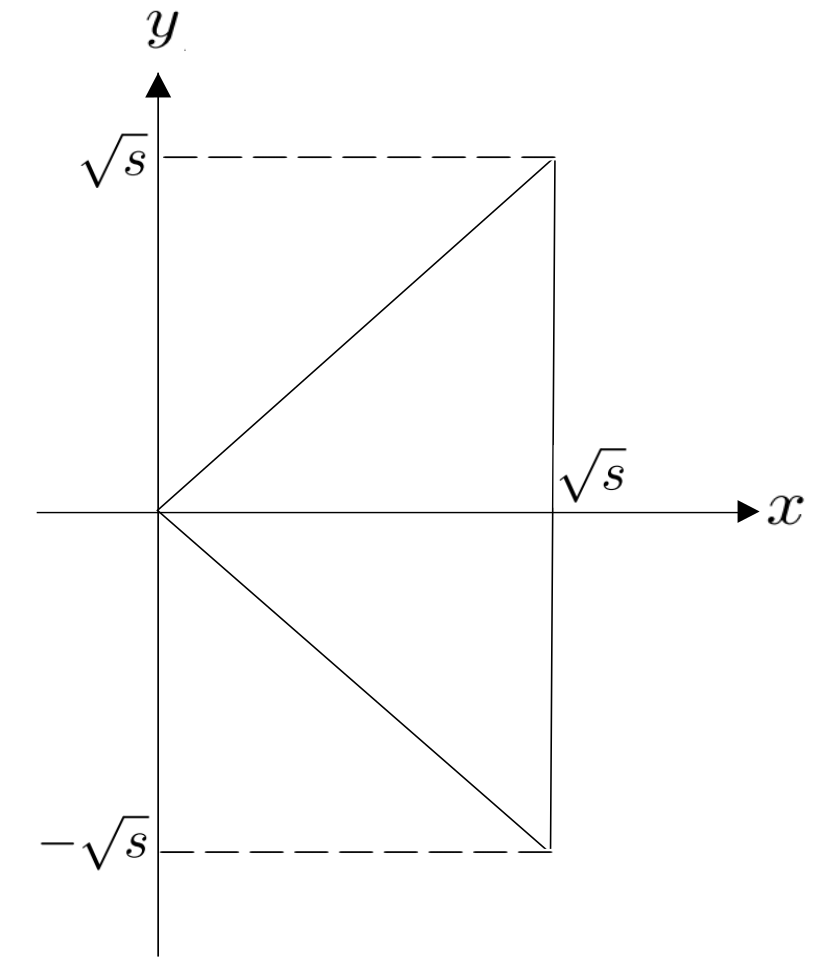}

Figure \ref{Shot1}: The boundaries of $A_{\epsilon,\delta}$ (left) and $\tau$ (right). 
\end{figure}

To see this, it is first useful to rewrite $S_N(x_0,s)$ using the constraint imposed on the summation over $m$ in (\ref{couting}) by the inner indicator function. Indeed, the constraint imposed on the inner sum is equivalent to
$$
(x_0 - m - \tfrac{1}{2N})^2 \leq n < (x_0 - m + \tfrac{1}{2N})^2,
$$
which amounts to 
$$
- \tfrac{1}{N}(x_0 - m) \leq n - (x_0 - m)^2 - (\tfrac{1}{2N})^2 < \tfrac{1}{N}(x_0 - m),
$$
giving us that 
$$
\chi_{[-\tfrac{1}{2} ,\tfrac{1}{2} )} ( N(\sqrt{n} - x_0 + m) ) = \chi_{[-1,1)} \bigg(\frac{N^{1/2}(n - (x_0 - m)^2 - (\tfrac{1}{2N})^2) }{N^{-\tfrac{1}{2}}(x_0 - m)} \bigg).
$$
Note also, $|\sqrt{n} - x_0 + m| \leq \tfrac{1}{2N}$, whenever $(m,n)$ contributes to the sum (\ref{couting}). So, the summation bound $1 \leq n \leq N'$ can be replaced by 
\begin{equation}\label{closeToRoot}
\chi_{(0,1]} \bigg( \frac{x_0-m+O(\tfrac{1}{2N})}{\sqrt{N'}} \bigg)
\end{equation}
giving us
$$
S_N(x_0,s) = \sum_{(m,n) \in \mathbb{Z}^2} \chi_{(0,1]} \bigg( \frac{x_0-m+O(\tfrac{1}{2N})}{\sqrt{N'}} \bigg)
\chi_{[-1,1)} \bigg(\frac{N^{1/2}(n - (x_0 - m)^2 - (\tfrac{1}{2N})^2) }{N^{-1/2}(x_0 - m)} \bigg)
$$
whenever $x_0 \neq 0$. The case $x_0 = 0$ can largely be ignored as the random variable $W_N$, which is uniformly distributed on the set $\Omega_N = \{k/N : 0 \leq k \leq N-1\}$, has probability $\tfrac{1}{N}$ of taking this value and we are interested in the limit as $N \to \infty$.

Therefore, the counting function can be bounded above and below using the following family of functions depending parameters $\epsilon$ and $\delta$, which can be realised as functions on $X$:
\begin{equation}\label{S_N adjust}
S_{N,\epsilon,\delta}(x_0,s) :=
\begin{cases}
\sum_{(m,n) \in \mathbb{Z}^2} \chi_{(-\epsilon,\sqrt{s}+\epsilon]} \bigg(\frac{x_0-m}{N^{1/2}} \bigg)
\chi_{[-1,1)} \bigg(\frac{N^{1/2}(n - (x_0 - m)^2) + \delta }{N^{-1/2}(x_0 - m)} \bigg) & \text{ if } x_0 \neq 0 \\
0 & \text{ if } x_0 = 0 
\end{cases}.
\end{equation}
Note that (\ref{closeToRoot}), together with the fact that $N' = \lfloor sN \rfloor$ implies we have
\begin{equation}\label{bounds on S_N}
S_{N,-\epsilon,\delta}(x_0,s) \leq S_N(x_0,s) \leq S_{N,\epsilon,\delta}(x_0,s)
\end{equation}
for $\epsilon = \epsilon_N := \tfrac{1}{2N(N')^{1/2}} + \big| \tfrac{(N')^{1/2}}{N^{1/2}} - \sqrt{s} \big|$, $\delta = \delta_N := - \tfrac{1}{4N^{3/2}}$ and $x_0 \neq 0$. As $N \to \infty$, the difference between the upper and lower bounds on $S_N(x_0,s)$ given by (\ref{bounds on S_N}) converges to zero in probability as $x_0$ runs over $\Omega_N$ according to $W_N$, as is shown in Proposition \ref{limit1}. 

The utility of introducing the functions $S_{N,\epsilon,\delta}$ is that they can be interpreted as functions of the form $f_A:X \to \mathbb{R}$ for suitable sets $A \subset \mathbb{R}^2.$

\begin{proposition}\label{triangleapprox}
\begin{equation}\label{YtoTriangle}
S_{N,\epsilon,\delta}(x_0,s) = f_{A_{\epsilon,\delta}} (\Gamma n(x_0)a(N))
\end{equation}
where 
$
A_{\epsilon,\delta} = A_{\epsilon,\delta}(s) := \{ (x,y) \in \mathbb{R}^2 : x \in (-\epsilon, \sqrt{s}+ \epsilon], \tfrac{y+\delta}{x} \in (-1,1] \}.
$
\end{proposition}
As one would expect, as $\epsilon$ and $\delta$ converge to zero, the domains $A_{\epsilon,\delta} = A_{\epsilon,\delta}(s)$ better and better approximate the triangle $\tau = \tau(s)$. This is shown in Figure \ref{Shot1}. 

\begin{proof}[Proof of Proposition \ref{triangleapprox}]
From (\ref{S_N adjust}), if we make the substitutions $(m,n) \longmapsto (-m,-n)$ and then $n \longmapsto n + m^2$ in the sum over $n$ we get 
$$
S_{N,\epsilon,\delta}(x_0,s) = \sum_{(m,n) \in \mathbb{Z}^2} \chi_{(-\epsilon,\sqrt{s}+\epsilon]} \bigg( \frac{x_0+m}{N^{1/2}} \bigg)
\chi_{[-1,1)} \bigg(\frac{N^{1/2}(n - x_0^2 + 2mx_0) + \delta }{N^{-\tfrac{1}{2}}(x_0 + m)} \bigg).
$$
To realise this as the value of a function of the space $X$ note for
$$
(M,x) := n(x_0)a(N) = 
\bigg(
\begin{pmatrix}
N^{-1/2} & 2x_0N^{1/2} \\
0 & N^{1/2}
\end{pmatrix}
, (\tfrac{x_0}{N^{1/2}},x_0^2 N^{1/2}) \bigg)
$$
we have that
$$
(m,n)M + x = \bigg(\frac{x_0+m}{N^{1/2}} , (2mx_0 + n + x_0^2)N^{1/2}\bigg).
$$
Thus,
$$
S_{N,\epsilon,\delta}(x_0,s) = f_{A_{\epsilon,\delta}} (\Gamma n(x_0)a(N))
$$
where 
$
A_{\epsilon,\delta} = A_{\epsilon,\delta}(s) := \{ (x,y) \in \mathbb{R}^2 : x \in (-\epsilon, \sqrt{s}+ \epsilon], \tfrac{y+\delta}{x} \in (-1,1] \}.
$
as required.
\end{proof}

To relate the random variables $Y_s^N$ to those defined on the space $X$ we set $\widetilde{S}_N(x_0,s) := f_{\tau(s)} (\Gamma n(x_0)a(N))$, $\widetilde{Y}_s^N := \widetilde{S}_N(W_N,s)$ and, more generally,  $Y^{N,\epsilon_N,\delta_N} := S_{N,\epsilon,\delta}(W_N,s)$. As we will now see, the limiting distribution of the variables $Y_s^N$ is identical to that of $\widetilde{Y}_s^N$. To see this, we first show:

\begin{proposition}\label{limit1}
$\mathbb{P}(Y^N_s \neq Y^{N,\epsilon_N,\delta_N}) \to 0$ as $N \to \infty$,
\begin{proof}
In light of (\ref{bounds on S_N}), it is sufficient to prove
$$
\mathbb{P}(Y^{N,-\epsilon_N,\delta_N}_s < Y^{N,\epsilon_N,\delta_N}_s) \to 0 
$$
as $N \to \infty$. Note 
\begin{align*}
S_{N,\epsilon_N,\delta_N}(x_0,s) - S_{N,-\epsilon_N,\delta_N}(x_0,s) & = 
(f_{A_{\epsilon_N,\delta_N}} - f_{A_{-\epsilon_N,\delta_N}})(\Gamma n(x_0)a(N))) \\& = f_{A_{N}} (\Gamma n(x_0)a(N))
\end{align*}
where $ A_{N} := A_{\epsilon_N,\delta_N} \setminus A_{-\epsilon_N,\delta_N}$. Now, we let the set $A'_{N}$ be the union the two rectangles $[-\epsilon_N , \epsilon_N] \times [-\eta,\eta]$ and $[1 - \epsilon_N, 1 + \epsilon_N] \times [-\eta,\eta]$, where $\eta := \sqrt{s} +1$. Then, for $N$ sufficiently large, $A_{N} \subset A'_{N}$ and $A'_{N} \searrow A_{\infty} := \{0,1\} \times [-\eta,\eta]$ which has (Lebesgue) measure zero. Therefore, whenever $n \leq N$ are sufficiently large,
$$
\mathbb{P}(Y^{N,-\epsilon_N,\delta_N}_s < Y^{N,\epsilon_N,\delta_N}_s) \leq \nu_N(f_{A'_{n}}).
$$
Each of the functions $f_{A'_{n}}$ is piecewise continuous as discontinuities of $f_{A'_{n}}$ correspond to lattices with points in the boundary of $A'_{n}$, which has (Lebesgue) measure zero. So, taking $\limsup_{N\to \infty}$ and using Theorem $\ref{main}$ we get 
$$
\limsup_{n \to \infty} \mathbb{P}(Y^{N,-\epsilon_N,\delta_N}_s < Y^{N,\epsilon_N,\delta_N}_s) \leq \int_X f_{A'_{n}} \; dm_X
$$
for all $n$ sufficiently large. Taking $n \to \infty$ and applying the dominated convergence theorem (as each set $A'_n$ is uniformly bounded and hence $f_{A'_{n}}$ is uniformly bounded by an integrable function for $n$ sufficiently large) we have $\int_X f_{A'_{n}} \; dm_X \to \int_X f_{A_{\infty}} \; dm_X$. By Lemma \ref{Integral0}, $\int_X f_{A_{\infty}} \; dm_X = 0$ which shows the required result.
\end{proof}
\end{proposition}
The above then allows us to prove the following.
\begin{proposition}\label{limit2}
$\mathbb{P}(Y^N_s \neq \widetilde{Y}_s^N) \to 0$ as $N \to \infty$
\begin{proof}
This goes along similar lines to the proof of Proposition \ref{limit1}. First note that, by this result, it is sufficient to show
$$
\mathbb{P}(Y^{N,\epsilon_N,\delta_N}_s \neq \widetilde{Y}_s^N ) \to 0
$$
as $N \to \infty$. Points in $\Omega_N$ where these random variables differ correspond to lattices with points in exactly one of the sets $\tau$ or $A_{\epsilon_N,\delta_N}$. Hence, $
\mathbb{P}(Y^{N,\epsilon_N,\delta_N}_s \neq \widetilde{Y}_s^N ) \leq \nu_N(f_{\tau \Delta A_{\epsilon_N, \delta_N} }).$
We can also find a sequence $\{\tau_N\}_{N=1}^{\infty}$ of regions in $\mathbb{R}^2$, each consisting of a triangular region with a smaller triangular region removed from its interior, such that $\tau \Delta A_{\epsilon_N, \delta_N, L} \subset \tau_N$ for all $N$ and $\tau_N \searrow W$, where $W$ has (Lebesgue) measure $0$. By taking limsups and using Theorem \ref{main}, we get 
$$ \limsup_{N \to \infty} \mathbb{P}(Y^{N,\epsilon_N,\delta_N}_s \neq \widetilde{Y}_s^N )  \leq \int_X f_{\tau_n} \; dm_X $$ for all $n$. Taking $n \to\infty$, using the dominated convergence theorem and Lemma \ref{Integral0} again gives the result. 

\end{proof}
\end{proposition}
Using these two propositions, we are now in a position to prove Theorems \ref{secondPigeon} and \ref{mainPigeon} using Theorem \ref{main}.

\begin{proof}[Proof of Theorem \ref{secondPigeon}]
Let $j \in \mathbb{N}_0$ and $s > 0$. 
By Proposition \ref{limit2}, it is enough to show 
$$ 
\lim_{N \to \infty} \mathbb{P}(\widetilde{Y}_s^N = j)
$$
exists. We have $\mathbb{P}(\widetilde{Y}^N_s = j) = \frac{1}{N} \sum_{k=0}^{N-1} \chi_{\{x_0: \widetilde{S_{N}}(x_0) = j\}}(\tfrac{k}{N}) = \nu_N(f_{j,s})$ where 
$f_{j,s} := \chi_{\{f_\tau(s) = j\}}$. Now, note that the points of discontinuity of $\chi_{\{f_{\tau(s)} = j\}}$ correspond to lattices with points in the boundary of the set $\tau(s)$, $\partial \tau(s)$. Namely, if $\Gamma(M, x)$ is a discontinuity point of $f_{j,s}$ then the lattice $ \{mM + x \}_{m \in \mathbb{Z}^2}$ contains a point in $\partial \tau$. But then $f_{\partial \tau(s)}(\Gamma(M, x)) \geq 1$. By Markov's inequality and Lemma \ref{Integral0}, the set of all such discontinuity points is contained in a set of measure zero, namely the set $\{f_{\partial \tau(s)} \geq 1 \}$. So we can apply Theorem \ref{main} which gives us that 
$$ \lim_{n \to \infty} 
\mathbb{P}(\widetilde{Y}^N_s = j) = \int f_{j,s} \; dm_X .$$
We hence see  the limiting distribution $E_j(s)$ of the quantities $E_{N,j}(s)$ is given by
\begin{equation}\label{LimitEquation}
E_j(s) = m_X( \{\Gamma(M,x) \in X: |(\mathbb{Z}^2M + x)\cap \tau(s)| = j \} ).
\end{equation}
\cite[Proposition 8.13]{marklof2010distribution} immediately tells us this function is $C^2$.
\end{proof}

\begin{proof}[Proof of Theorem \ref{mainPigeon}]
Let $Y_s:X \to \mathbb{R}$ be given by $Y_s = f_{\tau(s)}$ and let $\xi$ be the associated point process. Given a point $\Gamma (M,x) \in X$, this point process takes the form 
$$
\xi(\Gamma (M,x)) = \sum_{j=1}^{\infty} \delta_{s_j}
$$
where $s_j = \inf\{s>0 : Y_s(\Gamma (M,x)) \geq j\}$. In this setting we have, analogously to (\ref{processIntervalRelation2}), that $\xi((a,b]) = Y_b - Y_a$.   This agrees with with (\ref{detlaSum}) due to the definition of $\tau(s)$.  If $s_j = s_{j+1}$ for some $j$, it must be the case that the lattice defined by $(M,x)$ contains multiple points on the boundary of the triangle $\tau(s_j)$. Now, for any $s >0$, the boundary of the triangle $\tau(s)$ is contained within the lines $y=x$, $y = -x$ and $x = \sqrt{s}$. So, if the lattice defined by $(M,x) \in X$ does intersect the boundary of the triangle $\tau(s)$ in a set of size at least 2 for some $s$, then either 
\begin{itemize}
\item the lattice contains a point in the line $y = x$.
\item the lattice contains a point in the line $y = -x$.
\item the lattice contains multiple points in the line $x = \sqrt{s}$ for some $s > 0$.
\end{itemize}
By Lemma \ref{Integral0}, the measure of the set of all points $\Gamma (M,x) \in X$ whose corresponding lattice intersects the lines $y=x$ or $y=-x$ is zero. Moreover, in the case where we have multiple lattice points on the line $x = \sqrt{s}$ for some $s >0$, we can find $(u,v) \in \mathbb{Z}^2$ such that $(u,v)M$ has first coordinate equal to $0$. So, for $(u,v) \in \mathbb{Z}^2$, define 
$$G_{u,v} := \{ (M,x) \in G: (u,v)M = (0,y) \text{ for some } y \in \mathbb{R}\}.$$
For all $(u,v)$, $G_{u,v}$ is a codimension one submanifold of $G$ and so $m_X(G_{u,v}) =0$. Therefore, the set of all lattices with multiple points on one of the vertical lines $x = \sqrt{s}$ has measure zero. Consequently, the points $(s_j)_{j=1}^{\infty}$ are almost surely distinct and so the process $\xi$ is simple. 

To verify condition (i) in Lemma \ref{processConvergence} holds, take an interval $(a,b] \subset \mathbb{R}^+$. Then, $\mathbb{E}[\xi_N((a,b])] = \mathbb{E}[ Y^N_b - Y^N_a] \to b-a $ as $N \to \infty$ since, for any $s > 0$, the interval $[x_0 - \tfrac{1}{2N}, x_0 + \tfrac{1}{2N}) + \mathbb{Z} $ contains, on average, $s + O(\tfrac{1}{N})$ points from the sequence 
$(\sqrt{n})_{n=1}^{\lfloor sN \rfloor}$ as $x_0$ varies across $\Omega_N$. By Lemma \ref{Integral0}, $m_X( \xi((a,b])) = m_X(Y_b - Y_a) = b-a$.

To verify condition (ii) holds, let $k \in \mathbb{N}$ and $a_1 < b_1 \leq a_2 < b_2 \leq \dotsc \leq a_k < b_k$ be non-negative real numbers. Set $V = \cup_{j=1}^k (a_j,b_j]$. 
\begin{equation}\label{Eq1}
\mathbb{P}(\xi_N(V) = 0) = \mathbb{P}( \cap_{j=1}^k \{Y^N_{b_j} = Y^N_{a_j} \} )
\end{equation}
Now let $V_N := \cap_{j=1}^k (\{Y^N_{b_j} = \widetilde{Y}^N_{b_j} \} \cap \{Y^N_{a_j} = \widetilde{Y}^N_{a_j} \})$ 
By Proposition \ref{limit2}, $\mathbb{P}(V_N) \to 1$ as $N \to \infty$. Therefore, 
\begin{align}
&\limsup_{N \to \infty} |\mathbb{P}( \cap_{j=1}^k \{Y^N_{b_j} = Y^N_{a_j} \} ) - \mathbb{P}( \cap_{j=1}^k \{\widetilde{Y}^N_{b_j} = \widetilde{Y}^N_{a_j} \} ) |    \\
& \leq  \limsup_{N \to \infty} \; \mathbb{P}( \cap_{j=1}^k \{Y^N_{b_j} = Y^N_{a_j} \} \Delta \cap_{j=1}^k \{\widetilde{Y}^N_{b_j} = \widetilde{Y}^N_{a_j} \} )\\
& \leq \limsup_{N \to \infty} \mathbb{P}(B_N^c) = 0 \label{Eq2}
\end{align}
Moreover,
\begin{equation}\label{Eq3}
\mathbb{P}( \cap_{j=1}^k \{\widetilde{Y}^N_{b_j} = \widetilde{Y}^N_{a_j} \} ) = \nu_N(\chi_{\{f_D = 0\}} )
\end{equation}
where $D = D(a_1,b_1;a_2,b_2;\dotsc ;a_k,b_k)$ is the set $\cup_{j=1}^k (\tau(b_j) \setminus \tau(a_j))$. Note that the function 
$\chi_{\{f_D = 0\}}$ has discontinuities at points in $X$ whose corresponding lattice contains a point in the boundary of $D$. Since the boundary of $D$ is union of the boundaries of the triangles $\tau(a_j)$ and $\tau(b_j)$, it has measure zero. Thus, we can apply Theorem \ref{main} and deduce that 
\begin{equation}\label{Eq4}
\nu_N(\chi_{\{f_D = 0\}} ) \to m_X(f_D = 0)
\end{equation}
as $N \to \infty$. 
Finally, given 
\begin{equation}\label{Eq5}
m_X(\xi(V) = 0) = m_X(\cap_{j=1}^k \{f_{\tau(a_j)} = f_{\tau(b_j)} \}) = m_X(f_D = 0)
\end{equation}
we get 
$$
\lim_{N \to \infty} \mathbb{P}(\xi_N(V) = 0) = m_X(\xi(V) = 0)
$$
by combining (\ref{Eq1}), (\ref{Eq2}), (\ref{Eq3}), (\ref{Eq4}) and (\ref{Eq5}), completing the proof.
\end{proof}

The proof of Corollary \ref{SecondMoment} relies on the following consequence of the Siegel integral formula.

\begin{lemma}[{{\cite[(3.7)]{el2015distribution}}}]\label{siegel}
Let $F_1,F_2 \in L^1(\mathbb{R}^2)$. Then
$$
\int_X \sum_{m_1 \neq m_2 \in \mathbb{Z}^2} F_1((m_1M+x)F_2((m_2M+x)) dm_X(M,x) = \int_{\mathbb{R}^2}F_1 dx \int_{\mathbb{R}^2}F_2 dx.
$$
\end{lemma}
We will also use non-escape of results proved by El-Baz, Marklof and Vinogradov in \cite{el2015two}, which is the content of (\ref{LimLimsup}) below.

\begin{proof}[Proof of Corollary \ref{SecondMoment}]
Expanding the formula for $|Y_s(M,x)|^2$ we see that
\begin{align*}
    \int_X |Y_s(M,x)|^2 d m_X(M,x) =&  \int_X \sum_{m_1 \neq m_2 \in \mathbb{Z}^2} \chi_{\tau(s)}(m_1M+x)\chi_{\tau(s)}(m_2M+x)d m_X(M,x) \\
    + & \sum_{m \in \mathbb{Z}^2} \chi_{\tau(s)}(mM+x) d m_X(M,x)
\end{align*}
This equals $s^2 +s$ by Lemma \ref{siegel} and Lemma \ref{Integral0}.

As in \cite{el2015two}, we define
$$ \mathcal{P}_N := \{ \sqrt{n} + \mathbb{Z}: 1 \leq n \leq N \text{ and } n \text{ is not a square} \}. $$
Also, for an interval $I \subset \mathbb{R}$ we define the function $Z_N(I,\cdot): [0,1) \to \mathbb{R} $ by setting 
$$
Z_N(I,\alpha) := | (|\mathcal{P}_N|^{-1}I + \alpha + \mathbb{Z}) \cap \mathcal{P}_N |.
$$
This gives the number of points of $\mathcal{P}_N$ in the interval $I$ when normalized and shifted by $\alpha$. Now, equation (2.5) in \cite{el2015two} tells us
\begin{equation}\label{LimLimsup}
\lim_{R \to \infty} \limsup_{N \to \infty} \int_{\{Z_N(I,\cdot) > R\}} Z_N(I,\alpha) \; d\alpha = 0.
\end{equation}
Fix $s >0$. For any $N \in \mathbb{N}$, one can see from the inequality $\sqrt{t+1} - \sqrt{t} \geq \tfrac{1}{2\sqrt{t+1}} $ that all points $\mathcal{P}_{sN}$ lie a distance at least $\tfrac{1}{2\sqrt{Ns+1}}$ away from $0 \in \mathbb{T}$. As a consequence, when $N$ is sufficiently large, the only points of the sequence $\{ \sqrt{n} + \mathbb{Z}: 1 \leq n \leq sN \}$ which lie in the interval of width $\tfrac{1}{N}$ centred at $0$ are themselves $0$ and correspond to squares less than $sN$. Therefore, when $N$ is sufficiently large we have $S_N(0,s) = \lfloor \sqrt{sN} \rfloor $ and, if we define 
$ \widehat{E}_{j,N}(s) $ to be to be the proportion of the intervals $\{[x_0 - \tfrac{1}{2N} , x_0 + \tfrac{1}{2N}) + \mathbb{Z} : x_0 \in \Omega_N \}$ containing $j$ points of $\mathcal{P}_{sN}$, 
\begin{equation}\label{EvsEhat}
|E_{j,N}(s) - \widehat{E}_{j,N}(s)| \leq \frac{1}{N}.
\end{equation}
Now, for a large natural number $R$ we have that 
\begin{align*}
& \Big|\mathbb{E}[(Y^N_s)^2] - \int Y_s^2 \; dm_X -s \Big| \\
 \leq & \bigg| \frac{(S_N(0,s))^2}{N} - s \bigg| + \bigg| \frac{1}{N}\sum_{x_0 \in \Omega_N \setminus \{0\}} S_N(x_0,s)^2 - \int Y_s^2 \; dm_X \bigg| \\
 \leq & \bigg| \frac{(S_N(0,s))^2}{N} - s \bigg| + \bigg| \sum_{j=0}^R j^2 \widehat{E}_{j,N}(s) - \sum_{j=0}^{\infty} j^2 E_j(s) \bigg|
+ \sum_{j= R+1}^{\infty} j^2 \widehat{E}_{j,N}(s) .
\end{align*}
The first term above here clearly tends to $0$ as $N \to \infty$ whilst the second tends to $\sum_{j=R+1}^{\infty} j^2 E_j(s)$ as a consequence of Theorem \ref{secondPigeon} and (\ref{EvsEhat}). So, to complete the proof we need to show 
\begin{equation}\label{LimLimsup2}
\lim_{R \to \infty} \limsup_{N \to \infty} \sum_{j= R+1}^{\infty} j^2 \widehat{E}_{j,N}(s) = 0.
\end{equation}
Firstly, note 
\begin{align*}
\sum_{j = R+1}^{\infty} j^2 \widehat{E}_{j,N}(s) =& \frac{1}{N} \sum_{x_0 \in \Omega_N \setminus \{0\}} S_N(x_0,s)^2 \chi_{\{S_N(\cdot, s) \geq R+1\}}(x_0) \\
\leq & \frac{1}{N}\sum_{x_0 \in \Omega_N} |Z_{sN}([-\tfrac{s}{2},\tfrac{s}{2}),x_0)|^2 \chi_{\{Z_{sN}([-\tfrac{s}{2},\tfrac{s}{2}), \cdot ) \geq R +1 \}}(\tfrac{k}{N})
\end{align*}
since the shifted intervals $ |\mathcal{P}_{sN}|^{-1}[-\tfrac{s}{2}, \tfrac{s}{2}) + x_0 + \mathbb{Z}$ contain 
$[x_0 - \tfrac{1}{2N} , x_0 + \tfrac{1}{2N}) + \mathbb{Z}$. 

Secondly, for any $\alpha \in [x_0 - \tfrac{1}{2N} , x_0 + \tfrac{1}{2N})$, as the interval $ |\mathcal{P}_{sN}|^{-1}[-s, s) + \alpha + \mathbb{Z}$ contains $ |\mathcal{P}_{sN}|^{-1}[-\tfrac{s}{2}, \tfrac{s}{2}) + x_0 + \mathbb{Z}$, we have  $Z_{sN}([-\tfrac{s}{2},\tfrac{s}{2}),x_0) \leq Z_{sN}([-s,s),\alpha)$ for any such $\alpha$. Thus we get 
\begin{align*}
& \frac{1}{N}\sum_{x_0 \in \Omega_N} |Z_{sN}([-\tfrac{s}{2},\tfrac{s}{2}),x_0)|^2 \chi_{\{Z_{sN}([-\tfrac{s}{2},\tfrac{s}{2}), \cdot ) \geq R +1 \}}(x_0) \\
 \leq &\frac{1}{N}\sum_{x_0 \in \Omega_N} N \int_{x_0 - \tfrac{1}{2N}}^{x_0 + \tfrac{1}{2N}} |Z_{sN}([-s,s),\alpha)|^2 \chi_{\{Z_{sN}([-s,s), \cdot ) \geq R +1 \}}(\alpha) \; d\alpha \\
 = & \int_{\{Z_{sN}([-s,s), \cdot ) > R \}} |Z_{sN}([-s,s),\alpha)|^2 \; d\alpha
\end{align*}
Taking $N \to \infty$ and applying (\ref{LimLimsup}) gives us (\ref{LimLimsup2}) and hence the result.
\end{proof}

\section{Properties of the limiting process}\label{Proporties}
As we noted in \textsection \ref{intro}, $\xi$ is a simple intensity 1 process which does not have independent increments. The simplicity of this process was shown in the proof of Theorem \ref{mainPigeon}. The fact it has intensity 1 follows from Lemma \ref{Integral0} since for any interval $(a,b] \subset [0,\infty)$ we have that
$$
m_X[\xi((a,b])] = m_X(Y_b - Y_a) = b - a .
$$
To see that $\xi$ doesn't have independent increments consider the intervals $A := [0,2)$, $B := [2,\sqrt{5})$ and $C := [\sqrt{5},3)$. Then 
$$ m_X(\xi(A\cup C) \geq 1 \; | \; \xi(B) = 1) = 1$$
but 
\begin{equation}\label{NoLatticePoints}
m_X(\xi(A \cup C)) < 1.
\end{equation}
\begin{figure}[h]\label{Shot3}
\includegraphics[scale=0.45]{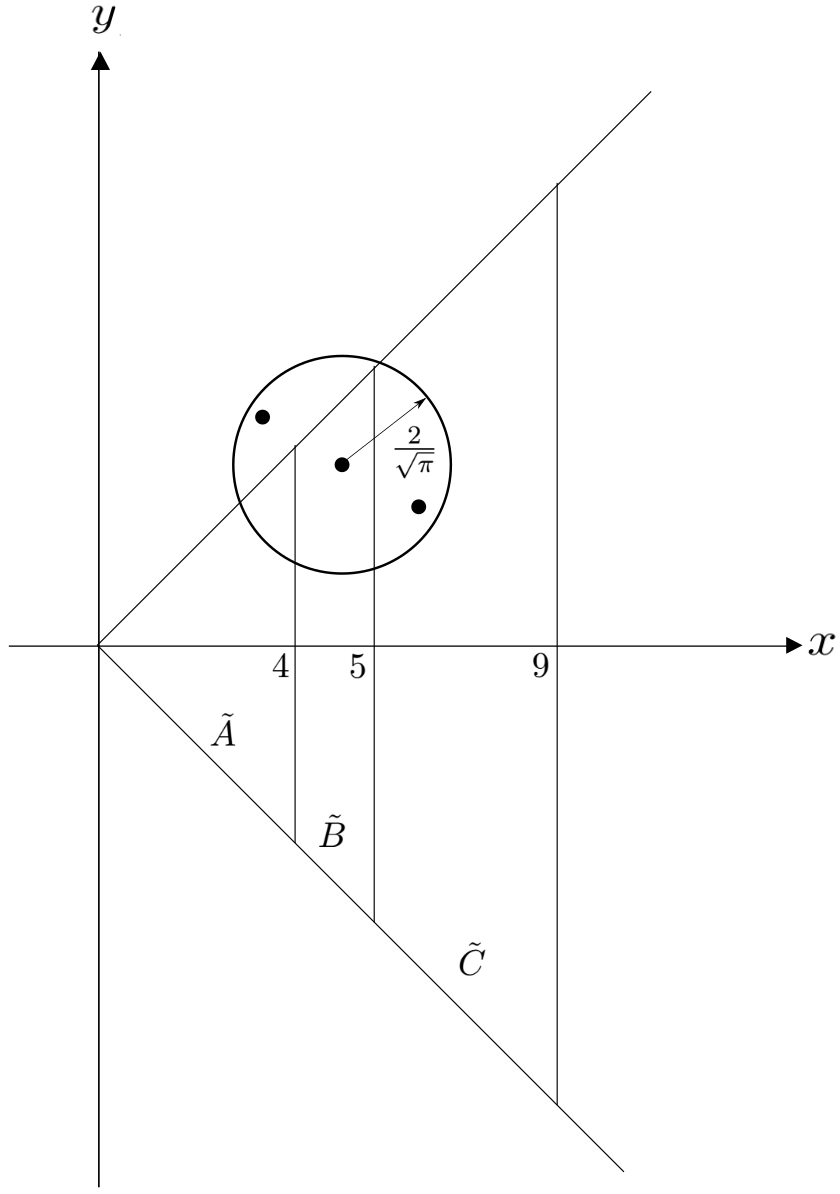} \hspace{1cm}\includegraphics[scale=0.55]{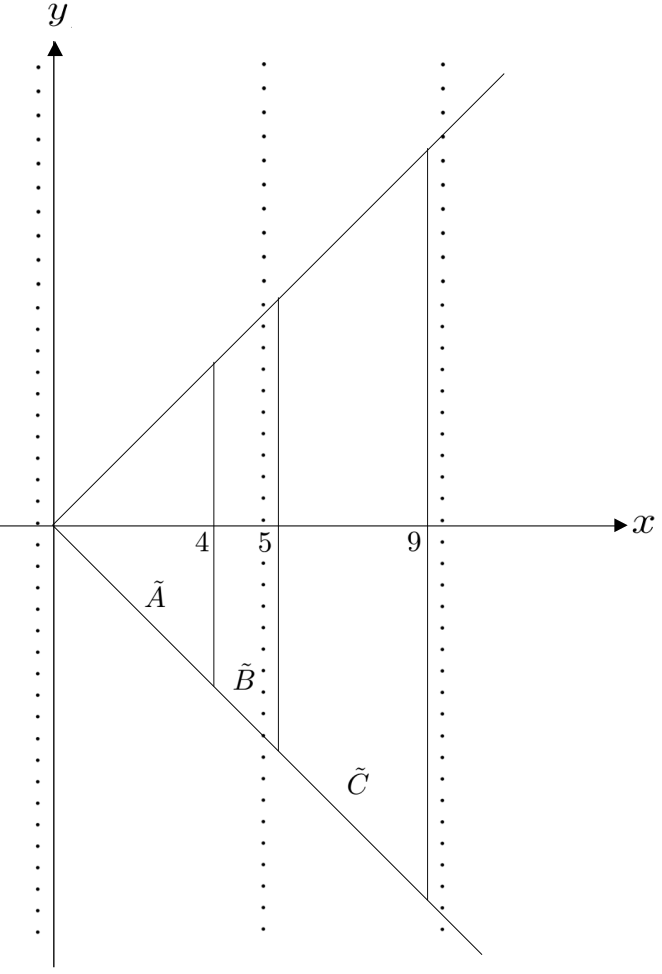}
Figure \ref{Shot3}: Any lattice containing a single point in $\widetilde{B}$ will contain one in either $\widetilde{A}$ or $\widetilde{C}$ (left). An example of a lattice with no points in $\widetilde{A} \cup \widetilde{C}$ (right). 
\end{figure}
This follows from the fact that, if there is exactly one point of the lattice given by $x \in X$ in the set $\widetilde{B} = 
\{ (u,v) \in \tau(\infty) | 4 \leq u < 5 \} $ then, by Minkowski's theorem, there is another lattice point in $\tau(\infty)$ which lies a distance at most a distance $\tfrac{2}{\sqrt{\pi}}$ away. Given $\xi(x)(B) =1$, this point must lie in either the set $ \widetilde{A} = \{ (u,v) \in \tau(\infty) | 0 \leq u < 4 \} $  or $\widetilde{C} = \{ (u,v) \in \tau(\infty) | 5 \leq u < 9 \} $ giving us $\xi(x)(A \cup B) \geq 1$. Conversely, it can easily been seen via analysing the form of the Haar measure on $X$ that a positive proportion of our affine unimodular lattices contain no points in $\widetilde{A} \cup \widetilde{B}$. An example of such a lattice is shown in figure \ref{Shot3}.

In terms of understanding the distribution of the points $\{\sqrt{n} + \mathbb{Z}: 1 \leq n \leq sN \}$ among our partition intervals, the lack of independent increments in the limiting point process tells us, when $(a,b] \cap (c,d] =\emptyset$, the points $ \{\sqrt{n} + \mathbb{Z}: aN < n \leq bN \}$ and $\{\sqrt{n} + \mathbb{Z}: cN < n \leq dN \}$ don't distribute among the partition intervals independently in the limit as $N \to \infty$. For example, when $S_N(x_0,2)$ is large, then, on average, $S_N(x_0,\sqrt{5}) - S_N(x_0,2)$ will be also. This is made intuitively clear by the fact that, if $\widetilde{A}$ contains many lattice points, then we would also expect $\widetilde{B}$ to do so also.

\section{Acknowledgements}
The author would like to thank Jens Marklof for his helpful guidance throughout the writing of this paper and the Heilbronn Institute for Mathematical Research for their support. Thanks should also be given to the anonymous referee for their comments and suggestions on the original version of this paper.

\printbibliography

\end{document}